\documentclass[a4paper]{amsart}

\pdfoutput=1

\usepackage[utf8]{inputenc}
\usepackage[T1]{fontenc}
\usepackage{lmodern}
\usepackage{amsthm, amssymb, amsmath, amsfonts, mathrsfs,upgreek}
\usepackage[colorlinks=true, pdfstartview=FitV, linkcolor=blue, citecolor=blue, urlcolor=blue,pagebackref=false]{hyperref}

\usepackage{microtype}

\newtheorem{thm}{Theorem}[section]
\newtheorem{prop}[thm]{Proposition}
\newtheorem{lem}[thm]{Lemma}
\newtheorem{cor}[thm]{Corollary}
\theoremstyle{remark}
\newtheorem{rem}[thm]{Remark}
\newtheorem{obs}[thm]{Observation}

\renewcommand{\le}{\leqslant}
\renewcommand{\ge}{\geqslant}
\renewcommand{\subset}{\subseteq}
\renewcommand{\emptyset}{\varnothing}
\newcommand{\mcl}{\mathcal}

\newcommand{\Ll}{\left}
\newcommand{\Rr}{\right}

\newcommand{\1}{\mathbf{1}}
\newcommand{\N}{\mathbb{N}}
\newcommand{\R}{\mathbb{R}}

\newcommand{\Z}{\mathbb{Z}}
\newcommand{\ov}{\overline}

\newcommand{\td}{\tilde}
\newcommand{\eps}{\varepsilon}
\renewcommand{\epsilon}{\varepsilon}

\renewcommand{\P}{\mathbb{P}}

\newcommand{\T}{\mathbb{T}}
\newcommand{\CT}{\mathbb{CT}}
\newcommand{\X}{\mathcal{W}}

\numberwithin{equation}{section}

\title[Phase transition of the contact process on random regular graphs]{Phase transition of the contact process on random regular graphs}
\author{Jean-Christophe Mourrat, Daniel Valesin}

\address[Jean-Christophe Mourrat]{ENS Lyon, CNRS, 46 allée d'Italie, 69007 Lyon, France}

\address[Daniel Valesin]{University of British Columbia, 1984 Mathematics Road, Vancouver BC, Canada V6T 1Z2}

\begin{document}

\begin{abstract}

We consider the contact process with infection rate $\lambda$ on a random $(d+1)$-regular graph with $n$ vertices, $G_n$. We study the extinction time $\uptau_{G_n}$ (that is, the random amount of time until the infection disappears) as $n$ is taken to infinity. We establish a phase transition depending on whether $\lambda$ is smaller or larger than $\lambda_1(\T^d)$, the lower critical value for the contact process on the infinite, $(d+1)$-regular tree: if $\lambda < \lambda_1(\T^d)$, $\uptau_{G_n}$ grows logarithmically with $n$, while if $\lambda > \lambda_1(\T^d)$, it grows exponentially with $n$. This result differs from the situation where, instead of $G_n$, the contact process is considered on the $d$-ary tree of finite height, since in this case, the transition is known to happen instead at the \emph{upper} critical value for the contact process on $\T^d$. %As part of the proof of the first statement, we show that if $\lambda < \lambda_1(\T^d)$, then the contact process with parameter $\lambda$ dies out on any graph with degree bounded by $d$.

\bigskip


\medskip


\end{abstract}
\maketitle
%
%
%
%
%
%%%%%%%%%%%%%%%%%%%%%%%%%%%%%%%%%%%%%%%%%%%%%%%%%%%%%%%%%%%%%%
%%%%%%%%%%%%%%%%%%%%%%%%%%%%%%%%%%%%%%%%%%%%%%%%%%%%%%%%%%%%%%
%
%
%
\section{Introduction}
\label{s:intro}
Let $G = (V, E)$ be a locally finite and connected graph. 
The contact process with parameter $\lambda > 0$ on a graph $G = (V, E)$ is a continuous-time Markov process on $\{0,1\}^V$ with generator $\mathcal{L}$ defined, for any local function $f: \{0,1\}^V \to \R$ and any configuration $\xi \in \{0,1\}^V$, by
$$(\mathcal{L}f)(\xi) = \sum_{x \in V} \Ll[f(\xi^{0\to x}) - f(\xi)\Rr] + \lambda \sum_{x,y\in V}\Ll[\1_{\{\xi(x) = 1\}} \cdot n(x,y) \cdot (f(\xi^{1 \to y}) - f(\xi)) \Rr],$$
where
$$\xi^{i \to x}(z) = \begin{cases} i &\text{if } z = x;\\\xi(z)&\text{otherwise},\end{cases}$$
$\1$ is the indicator function, and $n(x,y)$ is the number of edges between $x$ and $y$ (which we allow to be larger than 1).

Vertices of the graph are interpreted as individuals and states 0 and 1 indicate that an individual is healthy or infected, respectively. The transitions that appear in the definition of $\mathcal{L}$ are then understood as follows: $\xi \to \xi^{0 \to x}$ means that there is a \textit{recovery} at $x$, and $\xi \to \xi^{1 \to x}$ means that there is a \textit{transmission} to $x$ (this is necessarily triggered by a neighbour $y \sim x$ with $\xi(y) = 1)$. We refer the reader to \cite{lig99} for a thorough introduction to the contact process, including the facts that we state without proof in this introduction.

We often abuse notation and, for a set $S$, treat an element $\xi \in \{0,1\}^S$ as the same as the set $\{x\in S: \xi(x) = 1\}$; in particular, the configuration identically equal to zero is denoted by $\varnothing$. We denote by $(\xi^A_t)_{t \geq 0}$ the contact process on $G$ with initial configuration $\xi^A_0 \equiv A$. If $A = \{x\}$, we write $(\xi^x_t)$ instead of $(\xi^{\{x\}}_t)$. We write $P_{G,\lambda}$ for a probability measure under which the contact process with parameter $\lambda$ on $G$ is defined.% (the initial configuration will always be clear from the context). 

The \textit{extinction time} $\uptau^A_G$ for the contact process on $G = (V,E)$ with initial configuration $A \subset V$ is defined by
$$\uptau^A_G = \inf\{t: \xi^A_t = \varnothing\}.$$
Since the dynamics only allows for new infections to appear by transmission, the configuration $\varnothing$ is absorbing, so we have $\xi^A_t = \varnothing$ for any $t \geq \uptau^A_G$. We write $\uptau_G$ instead of $\uptau^G_G$ and $\uptau^x_G$ instead of $\uptau^{\{x\}}_G$.

When the contact process is considered on infinite graphs, a central question is whether we have survival or extinction of the infection. For a graph $G = (V, E)$, a finite set $A \subset V$ and $\lambda > 0$, define
$$\begin{aligned}&p_{G,A,\lambda}^{\text{ext}}= P_{G,\lambda}\Ll[\uptau_{G}^A < \infty\Rr]= P_{G,\lambda}\Ll[\exists t_0: \xi^A_t = \varnothing \text{ for all } t \geq t_0\Rr] ,\\
&p_{G,A,\lambda}^{\text{loc ext}} =  P_{G,\lambda}\Ll[\exists t_0: \xi^A_t \cap A = \varnothing \text{ for all } t \geq t_0\Rr],\end{aligned}$$
the probabilities of \textit{extinction} and \textit{local extinction} of the contact process with parameter $\lambda$ on $G$ started from $\xi^A_0 \equiv A$. It can be shown that %, depending on the value of $\lambda$ 
we either have $p^{\text{ext}}_{G,A,\lambda} = 1$ for all $A$ or $p^{\text{ext}}_{G,A,\lambda} < 1$ for all $A$; likewise, either $p^{\text{loc ext}}_{G,A,\lambda} = 1$ for all $A$ or $p^{\text{loc ext}}_{G,A,\lambda} < 1$ for all $A$. We say that the contact process
\begin{itemize}
\item \textit{dies out} if $p^{\text{ext}}_{G,A,\lambda} = 1$ for all $A$;
\item \textit{survives weakly} (or \textit{globally but not locally}) if $p^{\text{ext}}_{G,A,\lambda} < 1$ and $p^{\text{loc ext}}_{G,A,\lambda} = 1$ for all $A$;
\item \textit{survives strongly} (or \textit{locally}) if $p^{\text{loc ext}}_{G,A,\lambda} < 1$ for all $A$.
\end{itemize}
We let $\lambda_1(G) = \sup\{\lambda: p^{\text{ext}}_{G,A,\lambda} = 1\}$ and $\lambda_2(G) = \sup\{\lambda: p^{\text{loc ext}}_{G,A,\lambda} = 1\}$. It is well known that for $G = \Z^d$,  $0< \lambda_1(\Z^d) = \lambda_2(\Z^d) =:\lambda_c(\Z^d) < \infty$, whereas for $G = \T^d$, the infinite, $(d+1)$-regular tree (with $d \geq 2$), the situation is quite different, as we then have $0<\lambda_1(\T^d) < \lambda_2(\T^d) <\infty$. In the latter case, if we take $\lambda$ in the weak survival regime and start the process from a single infection at a vertex $x$, the infection has a chance of surviving, but it can only do so by propagating outwards from $x$; any finite neighbourhood of $x$ only carries the infection for a finite amount of time, as required in the definition of weak survival.

The contact process on finite graphs (deterministic or random) has also been the subject of much investigation; below we will survey some of the past work that is most relevant to the object of interest of this paper. Let us start noting that for finite graphs, there is no question of survival or extinction: if $G$ is finite, the infection almost surely disappears in finite time. One is thus interested in the behaviour of the infection before it disappears. The typical course of action goes as follows: we fix $\lambda$ and the parameters that define the graph, then take a size parameter (e.g.\ the number of vertices) to infinity, so as to obtain a sequence of graphs $G_n$, and consider the asymptotic behaviour of some random quantity $X_n$ associated with the contact process on $G_n$. Common choices for $X_n$ are the extinction time $\uptau_{G_n}$ and the average proportion of infected vertices before $\uptau_{G_n}$. Of particular interest are the cases in which the asymptotic behaviour of $X_n$ depends sensitively on the choice of $\lambda$ and the set of parameters that define the graph; this dependency can sometimes be related to a phase transition of the contact process on an infinite graph that is in some sense approximated by $G_n$.

This project has been first carried out for $G_n$ equal to the subgraph of $\Z^d$ induced by the vertices contained in a box of side length $n$; see %\cite{eulalia}, \cite{schonmeta}, \cite{durliu}, \cite{chencp}, \cite{dursc}, \cite{tommeta}, \cite{tomexp}, or Section I.3 of \cite{lig99} 
\cite{eulalia,schonmeta,durliu,chencp,dursc,tommeta,tomexp}, or \cite[Section I.3]{lig99} for an overview. As suggested in the previous paragraph, the behaviour of $\uptau_{G_n}$ exhibits a phase transition that mimics the phase transition of the contact process on $\Z^d$: $\uptau_{G_n}$ grows logarithmically or exponentially with the volume of $G_n$ respectively if $\lambda$ is taken smaller or larger than $\lambda_c(\Z^d)$.

%Let us now turn to finite trees. Let $\hat{\T}^d_\ell$ denote the subgraph of $\T^d$ induced by the vertex set given by a distinguished vertex $o$ (called the root) and all vertices at graph distance at most $\ell$ from $o$, and let $\T^d_\ell$ denote the sub-graph of $\hat{\T}^d_\ell$ where one of the sub-trees emanating from the root is removed. In \cite{St} and \cite{cmmv13}, the contact process on $\hat{\T}^d_\ell$ is studied (to be precise, these papers rather focus on $\T^d_\ell$, but their results and proofs are unaffected by this modification). In \cite{St}, bounds were obtained on $\uptau_{\hat{\T}^d_\ell}$ for different values of $\lambda$, and these bounds were improved in \cite{cmmv13} to yield the result below. We denote by $|A|$ the cardinality of a set $A$.

Let us now turn to finite trees. Let $\T^d_\ell$ denote the $d$-ary tree of height $\ell$. In other words, $\T^d_\ell$ is a tree with a distinguished vertex $o$ (called the root) so that $o$ has degree $d$, all vertices at distance between 1 and $\ell-1$ from $o$ have degree $d+1$ and all vertices at distance $\ell$ from $o$ have degree 1. In \cite{St} and \cite{cmmv13}, the contact process on ${\T}^d_\ell$ was studied. In \cite{St}, bounds were obtained on $\uptau_{{\T}^d_\ell}$ for different values of $\lambda$, and these bounds were improved in \cite{cmmv13} to yield the result below. We denote by $|A|$ the cardinality of a set $A$.

\begin{thm} \cite{cmmv13} 
\label{t:cut-trees}
Assume $d \geq 3$.
\begin{itemize}
\item[(a)] For any $\lambda \in (0,\; \lambda_2(\T^d))$, there exists $c \in (0,\infty)$ such that, as $\ell \to \infty$, $$\frac{\uptau_{ {\T}^d_{\ell}}}{\log | {\T}^d_\ell|} \to c \text{ in probability.}$$
\item[(b)] For any $\lambda \in (\lambda_2(\T^d),\infty)$, there exists $C \in (0,\infty)$ such that, as $\ell \to \infty$,
$$\frac{\log \uptau_{ {\T}^d_\ell}}{|{\T}^d_\ell|} \to C \text{ in probability.}$$
Moreover, $\uptau_{ {\T}^d_\ell}$ divided by its expectation converges in distribution to the exponential distribution with parameter 1.
\end{itemize}
\end{thm}
In short, the asymptotic behaviour of the extinction time on ${\T}^d_\ell$ has a phase transition at the value $\lambda = \lambda_2(\T^d)$. It is worth mentioning that the above theorem remains true if $\T^d_\ell$ is replaced by $\hat \T^d_\ell$, defined as the subgraph of $\T^d$ induced by the vertex set given by a distinguished root vertex $o$ and all vertices at graph distance at most $\ell$ from $o$ ($\T^d_\ell$ is obtained from $\hat{\T}^d_\ell$ by removing one of the sub-trees emanating from the root of $\T^d_\ell$).

The main goal of this paper is to investigate the asymptotic behaviour of the extinction time of the contact process on random $(d+1)$-regular graphs, which we now describe. Fix $n \in \N$, the number of vertices, with the restriction that $n(d+1)$ be even. Let us define a random graph $G_n$ with vertex set $V_n = \{1, \ldots, n\}$. We endow each vertex with $d+1$ half-edges; attaching two half-edges (say, one belonging to vertex $i$ and another to vertex $j$) produces an edge between $i$ and $j$. Then, any perfect matching on the set of half-edges produces a graph. We choose one perfect matching uniformly at random from all the $(n(d+1)-1)\cdot (n(d+1)-3)\cdots 3 \cdot 1$ possibilities, and this produces the random edge set $E_n$. Obviously, $G_n$ is a $(d+1)$-regular graph; by using an alternate construction of $G_n$ that is described in Section \ref{s:super}, it is easy to show that, for any $R$, with probability tending to 1 as $n \to \infty$, the ball of radius $R$ around a given vertex has no loop, and is thus isomorphic to $\hat{\T}^d_\ell$.

We write $\P_n$ for the law of $G_n$, and $\P_{n,\lambda}$ for a probability measure under which both $G_n$ and the contact process on $G_n$ with parameter $\lambda$ are defined. Our main result is
\begin{thm}\label{t:mainsup}\noindent Assume $d \geq 3$.
\begin{itemize}
\item[(a)] For every $\lambda \in (0, \lambda_1(\T^d))$ there exists $C < \infty$ such that
$$\lim_{n \to \infty} \P_{n,\lambda}\Ll[\uptau_{G_n} < C\log(n)\Rr] = 1.$$
\item[(b)] For every $\lambda \in (\lambda_1(\T^d), \infty)$ there exists $c > 0$ such that
$$\lim_{n\to \infty}\P_{n,\lambda}\Ll[\uptau_{G_n} > e^{cn}\Rr] = 1.$$  
\end{itemize}
\end{thm}

Let us stress the most important point: contrary to the situation in Theorem~\ref{t:cut-trees}, the behaviour of the extinction time has here a phase transition at $\lambda_1(\T^d)$ instead of $\lambda_2(\T^d)$. 
Although surprising at first, this phenomenon can be understood as follows:
%
%In particular, the order of magnitude of $\uptau_{\T^d_\ell}$ in the weak survival regime is the same as that for the extinction regime. This is not surprising given the comment we have made above concerning weak survival: the infection can only attempt to survive by distancing itself from the root, and on $\T^d_\ell$ this can only happen until the leaves are reached. 
%
when $\ell$ is large, the tree ${\T}^d_\ell$ \emph{as seen from a vertex chosen uniformly at random} does not at all locally look like $\T^d$. For example, with non-vanishing probability, the random vertex is a leaf of ${\T}^d_\ell$. The notion of local limit of a sequence of graphs, as formalized in \cite{bs}, captures such features, and one can check that ${\T}^d_\ell$ does not converge locally to $\T^d$, but rather to another infinite graph called the \textit{canopy tree} $\CT^d$ (see \cite[Example 5.14]{benj} for details). Moreover, one can show that 
\begin{equation}
\label{critct}
\lambda_1(\CT^d) = \lambda_2(\CT^d) = \lambda_2(\T^d)
\end{equation}
(local survival for $\lambda > \lambda_2(\T^d)$ can be derived from \cite[Theorem~1.6]{cmmv13}, while extinction for $\lambda \le \lambda_2(\T^d)$ follows from \cite[Proposition~4.57]{lig99}). In light of this, the logarithmic-exponential phase transition for the extinction time on ${\T}^d_\ell$ does happen at the \emph{lower} critical value of the limiting graph after all, but one must take into account that this limiting graph is the canopy tree.
%reflect an extinction-survival phase transition after all: that of the canopy tree.

It would be very interesting to find wider classes of graphs (such as the configuration model described below) for which one can prove that the same phase transition occurs at the lower critical value of the limiting graph. This is reminiscent of the question of the locality of the percolation critical probability, see \cite{bnp}, \cite[Section~5.2]{benj} and \cite{mt}.

%Later in this Introduction, we will recall a notion of local graph limit and then note that $\T^d_\ell$ does not converge locally to $\T^d$, but rather to another infinite graph called the \textit{canopy tree} (denoted $c\T^d$). We will also argue that $\lambda_1(c\T^d) = \lambda_2(\T^d)$. In light of this, the logarithmic-exponential phase transition for the extinction time on $\T^d_\ell$ does reflect an extinction-survival phase transition after all: that of the canopy tree.

%Apart from this discussion, it is natural to wonder if, for classes of graphs whose local limit actually is $\T^d$, $\lambda_1(\T^d)$ plays the role of threshold for a change of behaviour of the extinction time. The natural candidate is the random $d$-regular graph, which we now describe. 

The graph $G_n$ is a particular case of the class of graphs known as the \textit{configuration model}. Whereas for $G_n$ we assumed that all degrees are taken equal to $d+1$, in the configuration model one allows for a random degree sequence, typically i.i.d. from some fixed degree distribution. The contact process on these graphs has lately received attention (\cite{cd09,mmvy,mvy}). In particular, it was proved that, if the degree distribution is a power law (and the graph is assumed to be connected), then the extinction time grows exponentially with the number of vertices \textit{regardless of the value of $\lambda$}. In this sense, one can say that the contact process on the configuration model with a power law degree distribution has no phase transition: it is ``always supercritical''. Theorem \ref{t:mainsup} shows that, if the degrees are constant, then there is a phase transition. In fact, putting together Lemma \ref{lem:stdom} below with the main result of \cite{mmvy}, we get

\begin{thm}
Let $\mathcal{G}_{n,d}$ be the set of connected graphs with $n$ vertices and degree bounded by $d+1$.
\begin{itemize}
\item[(a)] For any $\lambda \in (0, \lambda_1(\T^d))$, there exists $C < \infty$ such that 
$$\lim_{n \to \infty} \inf_{G \in \mathcal{G}_{n,d}} P_{G,\lambda}\Ll[\uptau_G < C \log n\Rr] = 1.$$
\item[(b)] For any $\lambda \in (\lambda_c(\Z), \infty)$, there exists $c > 0$ such that
$$\lim_{n \to \infty} \inf_{G \in \mathcal{G}_{n,d}} P_{G,\lambda}\Ll[\uptau_G > e^{cn}\Rr] = 1.$$
\end{itemize}
\end{thm} 
In particular, if we take the contact process on the configuration model with a degree distribution with \textit{bounded support}, the extinction time exhibits a phase transition on $\lambda$. An open problem that we believe to be of much interest is whether such a phase transition also occurs if the degree distribution has unbounded support, but a light tail. Equally interesting (and difficult) is the corresponding question for supercritical Erd\H{o}s-R\'enyi random graphs. 

\medskip

We now comment on the proofs of the two parts of Theorem \ref{t:mainsup}. Part (a) relies on the fact, stated in Lemma \ref{lem:stdom}, that the contact process on any graph of degree bounded by $d+1$ is stochastically dominated (in terms of the number of infected vertices) by a contact process with same  $\lambda$ on $\T^d$. Although this is a very elementary and natural statement, we did not find it in the literature, so we give a proof, which relies on the concept of universal covering of a graph. Once Lemma \ref{lem:stdom} has been established, we invoke known bounds for the extinction time of the contact process on $\T^d$ in the extinction regime in order to conclude the proof.

As a side note, we remark that Lemma \ref{lem:stdom} also applies to infinite graphs, and as a consequence we prove:
\begin{thm}
\label{t:univsub}
If $\lambda \leq \lambda_1(\T^d)$, then the contact process with parameter $\lambda$ on any graph with degree bounded by $d$ dies out.
\end{thm}

The proof of part (b) of Theorem \ref{t:mainsup} is much more involved,  and relies on a geometric property of $G_n$ (Theorem~\ref{t:regen}). The idea can be roughly summarised as follows. Assume that at time $0$, there are $\eps n$ infected vertices ($\eps > 0$ small) with the property that within distance $r$ of each of these vertices, one can find a copy of $\T^d_\ell$, the copies (and the paths to them) being pairwise disjoint. Let us say that a set satisfying this property is \emph{regenerative}. If $\ell$ can be chosen sufficiently large, then after some time, the infection will spread to $k \eps n$ vertices with $k$ large, except for an event whose probability is exponentially small in $n$. The crucial geometric property of $G_n$ that we need is that if $k$ is sufficiently large, then with probability tending to $1$, \emph{any} subset of size $k\eps n$ contains a regenerative subset of size $\eps n$. This enables to iterate the argument. Since the probability of failure of a given step is exponentially small, the extinction time must be exponentially long.

\medskip

%There exist $\epsilon, \epsilon', r, \ell > 0$ which do not depend on $n$ and such that, with high probability as $n \to \infty$, $G_n$ satisfies the following. For any set  of vertices $A$ with $|A| = \epsilon n$, we can find vertices $y_1, \ldots, y_{\epsilon'n} \in A$ and subgraphs of $T_{y_1},\ldots,T_{y_{\epsilon'}n} \subset G_n$ so that the $T_{y_i}$ are all disjoint copies of $\T^d_\ell$ and the graph distance between $y_i$ and the root of $T_{y_i}$ is less than $r$. With this at hand, we assume that at time 0 the vertices of $A$ are infected and use large deviation estimates for the binomial distribution to argue that the following happens outside of probability exponentially small in $\epsilon'n$. For some subset $\{z_{1},\ldots, z_{\sigma n}\} \subset A'$, with $\sigma < \epsilon'$, the infection initially present at each of the $z_i$'s reaches the root of $T_{z_i}$, and from there propagates in the direction of the leaves, so that at some fixed time $S$, we can find a large number of infected vertices inside each $T_{z_i}$. We show that the resulting total number of infected vertices at time $S$ is larger than $\epsilon n$, and then we redefine $A$ to be this new set of infected vertices and repeat.

In the rest of this introduction, we give a summary of the notation we use and an exposition of the graphical construction of the contact process.
% and finally some comments regarding the contact process on the canopy tree. 
Sections \ref{s:super} and \ref{s:subcrit} are devoted to proving parts (b) and (a) of Theorem \ref{t:mainsup} respectively.

\subsection{Summary of notation for sets and graphs.}
For a set $A$, we write $|A|$ for the cardinality of $A$ and $\1_A$ for the indicator function of $A$.

A \emph{graph} is a pair $G = (V, E)$, where $V$ is the set of \emph{vertices} and $E$ is the set of \emph{edges}. 
With each edge is associated an unordered pair of points in $V$, the end-points of the edge. Loops (i.e.\ edges between a vertex and itself) and multiple edges with the same end-points are allowed. We may abuse notation and write $v \in G$ instead of $v \in V$ when convenient. The \textit{degree} of a vertex is the number of non-loop edges that contain it plus twice the number of loops that contain it. The \textit{graph distance} between vertices $x$ and $y$ is the length of a shortest path between $x$ and $y$, and is denoted by $\mathsf{dist}(x,y)$, or $\mathsf{dist}_G(x,y)$ when we want to make the graph explicit.

The set $\vec{E}$ of oriented edges of $G$ is defined as follows: for each $e \in E$, we add to $\vec{E}$ two oriented edges corresponding to the two possible orientations of $e$. A generic element of $\vec{E}$ will be denoted by $\vec{e}$. Note that, if there are $k$ edges containing $x$ and $y$, then there are $2k$ oriented edges containing $x$ and $y$. We write $v_0(\vec{e})$ and $v_1(\vec{e})$ to denote the starting and ending vertex of $\vec{e}$, respectively. We also let $u(\vec{e}) \in E$ denote the unoriented edge to which $\vec{e}$ is associated.

A \emph{rooted graph} is a pair $(\rho,G)$ where $G$ is a graph and $\rho \in G$. Given two rooted graphs $(\rho, G)$ and $(\rho', G')$ with $G = (V, E)$, $G' = (V', E')$, we say that $f : V \to V'$ is an \emph{embedding} of $(\rho,G)$ into $(\rho',G')$ if 
\begin{itemize}
\item $f(\rho) = \rho'$,
\item $f$ is injective,
\item for every $u,v \in V$, $u$ and $v$ are neighbours in $G$ if and only if $f(u)$ and $f(v)$ are neighbours in $G'$. 
\end{itemize}
We say that $(\rho',G')$ \emph{embeds} $(\rho,G)$ if there exists an embedding of $(\rho, G)$ into $(\rho',G')$; that $(\rho,G)$ and $(\rho',G')$ are \emph{isomorphic} if there exists a bijective embedding from one to the other.

For the rest of the paper, we fix $d \geq 3$ and omit $d$ from $\T^d$ and $\T^d_\ell$, thus writing $\T$ and $\T_\ell$ respectively. The root of $\T_\ell$ is denoted by $o$.

\subsection{Graphical construction of the contact process.}
Almost every paper on the contact process contains a brief exposition of its graphical construction, so this is by now quite a redundant addition. We nevertheless include it here because we allow our graphs to contain loops and parallel edges, so we need to be careful with the notation to avoid confusion.  

Fix a graph $G = (V,E)$ and an infection rate $\lambda > 0$. Assume given the following families of independent Poisson point processes on $[0,\infty)$: $(T_{\vec{e}})_{\vec{e} \in \vec{E}}$, all with rate $\lambda$, and $(D_x)_{x \in V}$, all with rate 1.  We now imagine $G$ is embedded on the $xy$-plane and add marks on the vertical lines $(V\cup \vec{E}) \times [0,\infty)$ as follows: for each $x$ and each $t \in D_x$, we add a so-called \textit{recovery mark} at $(x,t)$, and for each $\vec{e}$ and each $t \in T_{\vec{e}}$, we add a so-called \textit{transmission arrow} at $(\vec{e},t)$.

 Given $x,y \in V$ and $0 \leq s \leq t$, we say that $(x,s)$ and $(y,t)$ are connected by an \textit{infection path} (and write $(x,s) \leftrightarrow (y,t)$ if there exists a right-continuous function $\gamma:[s,t] \to V$ such that \begin{itemize}
\item $\gamma(s) = x,\;\gamma(t) = y$,
\item  $(r,\gamma(r)) \notin  D_{\gamma(r)}$ for all $r \in [s,t]$,
\item $\gamma(r-) \neq \gamma(r)$ implies $r \in T_{\vec{e}}$ for some $\vec{e}$ with $v_0(\vec{e}) = \gamma(r-),\;v_1(\vec{e}) = \gamma(r)$.
\end{itemize}
In other words, $(r,\gamma(r))_{s\leq r \leq t}$ must be a path from $(x,s)$ to $(y,t)$ which does not cross any recovery mark and is allowed to traverse transmission arrows.

For $x \in V$, let $\xi^x_t = \{y \in V: (x,0) \leftrightarrow (y,t)\}$; for $A \subset V$, let $\xi^A_t = \cup_{x \in A} \;\xi^x_t$. Then, $(\xi^A_t)_{t\geq 0}$ is a Markov process on $\{0,1\}^V$ with the same distribution as the one given by the generator $\mathcal{L}$ defined earlier. The graphical construction has the advantage that all the contact processes $((\xi^A_t)_{t\geq 0})_{A \subset V}$ are defined in the same probability space and, if $A \subset B$, then $\xi^A_t \subset \xi^B_t$ for all $t$. Note also that $\xi^A_{s+t} = \cup_{x\in \xi^A_{t}} \{y:(x,t) \leftrightarrow (y,t+s)\}$.

%\subsection{Comments about the contact process on the canopy tree}

\section{Supercritical regime}

The goal of this section is to prove part (b) of Theorem~\ref{t:mainsup}. We first state for later usage a classical large deviation result on binomial random variables, whose proof is standard.
\begin{lem}
\label{l:largedev}
Let $\mathsf{Bin}(m,p)$ denote a binomial random variable with parameters $m \in \N$ and $p \in [0,1]$ (defined with respect to the measure $\P$, say). For every $\delta \ge 0$, 
$$
\P[\mathsf{Bin}(m,p) \ge (p+\delta) m] \le e^{-m \psi_p(\delta)},
$$
where
\begin{eqnarray}
\label{defpsip}
\psi_p(\delta) &  = & \sup_{\lambda} \Ll[ \lambda (p+\delta) - \log(1-p+pe^\lambda )  \Rr] \notag \\
& = & (p+\delta) \log\Ll( \frac{p+\delta}{p} \Rr) + (1-p-\delta) \log\Ll( \frac{1-p-\delta}{1-p} \Rr) .
\end{eqnarray}
\end{lem}

We now recall the following important estimate from \cite{ss98}.

\label{s:super}
\begin{lem}\cite{ss98}
\label{l:ss}
For every $\lambda > \lambda_1(\T)$, there exists $S$, $p_0 > 0$ and $\alpha > 1$ such that for every $\ell$ large enough,
$$
P_{\T,\lambda}\Ll[ \left| \{x \in \xi^o_{\ell S}: \mathsf{dist}(o,s) = \ell \} \right|  \ge \alpha^\ell  \Rr] \ge p_0,
$$
\end{lem}

We will need a slight modification of the above result:

\begin{lem}
\label{l:ss1}
For every $\lambda > \lambda_1(\T)$ and $r> 0$, there exist $R$, $\sigma > 0$ and $\alpha > 1$ such that for every $\ell$ large enough, the following holds. For any graph $G$ with vertices $x, y$ so that $\mathsf{dist}(x,y) \leq r$ and $(y,G)$ embeds $(o,\T_\ell)$, 
$$P_{G,\lambda}\Ll[|\xi^{\{x\}}_{R\ell}| \geq \alpha^\ell\Rr] > \sigma.$$
\end{lem}
\begin{proof}
This follows from Lemma \ref{l:ss} and the facts that $$P_{G,\lambda}\Ll[\xi^{\{x\}}_t(x) = 1\Rr] \geq e^{-t},\qquad P_{G,\lambda}\Ll[\xi^{\{x\}}_{\mathsf{dist}(x,y)}(y) = 1\Rr] \geq \left(e^{-2}(1-e^{-\lambda})\right)^{\mathsf{dist(x,y)}}.$$ The first estimate is obtained by considering the event that $x$ does not recover, by time $t$, of the infection it has at time $0$. The second estimate is obtained by considering the event that, in every unit time interval $[k, k+1]$ from time 0 to time $t= \mathsf{dist}(x,y)$, the infection advances one step along a shortest path from $x$ to $y$ (given a pair of vertices $u, v$ such that $u$ is infected at time $k$, $v$ is sure to be infected at time $k+1$ if, in this unit time interval, neither of them recover and there is at least one transmission from $u$ to $v$).
\end{proof}
We say that a set of vertices $W \subset V_n$ is $(\ell,r)$\emph{-regenerative} if there exists a family $(G_v')_{v \in W}$ of subgraphs of $G_n$ that are pairwise disjoint and such that for every $v \in W$, the following two conditions hold:
\begin{itemize}
\item $G_v'$ contains $v$,
\item there exists $x \in G_v'$ such that the distance in $G_v'$ between $x$ and $v$ is $r$ and $(x,G_v')$ embeds $(o,\T_\ell)$.
\end{itemize}
The crucial geometric property of $G_n$ that we need is 
\begin{thm}[Finding large regenerative subsets]
For $k$ and $r$ sufficiently large and for every $\ell$, there exists $\eps_0 > 0$ such that for every $\eps \le \eps_0$, the following holds with $\P_n$-probability tending to~$1$ as $n$ tends to infinity. From every $W \subset V_n$ of cardinality at least $k \eps n$, one can extract an $(\ell,r)$-regenerative subset of cardinality at least $\eps n$.
\label{t:regen}
\end{thm}

\begin{rem} In what follows, in order to prevent the notation from getting too heavy, we will pretend that certain quantities, such as $\epsilon n$ and $k \epsilon n$ in the above proposition, are integers. It should be clear that in a correct but overscrupulous writing, one should take the pertinent integer parts or add or subtract 1 at certain places, but that this does not change our proofs in any relevant way.
\end{rem}

We prove part (b) of Theorem~\ref{t:mainsup} using Theorem~\ref{t:regen}, before turning to the proof of the latter.

\begin{proof}[Proof of part (b) of Theorem~\ref{t:mainsup}]
We fix $\lambda > \lambda_1(\T)$, and choose constants according to the following steps:
\begin{enumerate}\item fix $r$ and $k$ large, as required by Theorem~\ref{t:regen};
\item let $\alpha,R,\sigma$ correspond to $\lambda$ and $r$, as in Lemma \ref{l:ss1}; 
\item take $\ell$ large enough, as required by Lemma \ref{l:ss1}, and also so that $\alpha^\ell > \frac{2k}{\sigma}$;
\item take $\epsilon \leq \epsilon_0$, where $\epsilon_0$ corresponds to $k, r, \ell$ as in Theorem~\ref{t:regen}.
\end{enumerate}
Assume $G_n$ satisfies the property stated in Theorem~\ref{t:regen}, namely
$$\text{every }W \subset V_n \text{ with }|W| \geq k\epsilon n \text{ has an $(\ell,r)$-regenerative subset of cardinality $\epsilon n$}.$$
We will now prove that, for some constant $c > 0$ which does not depend on $n$,
\begin{equation}
\label{eq:equiv}\text{for all }W \subset V_n \text{ with } |W| \geq k\epsilon n,\; P_{G_n}\Ll[|\xi^W_{R\ell}| \geq k\epsilon n\Rr] \geq 1 - e^{-cn}.
\end{equation}
This will imply the statement of Theorem \ref{t:mainsup}.

We fix $W$ with $|W|\geq k\epsilon n$ and extract from it an $(\ell,r)$-regenerative subset $W'$ of cardinality $\epsilon n$; we enumerate its elements, $W'=\{v_1,\ldots, v_{\epsilon n}\}$. By the definition of $(\ell,r)$-regenerative sets, there exist pairwise disjoint subgraphs of $G$, $G'_{v_1},\ldots, G'_{v_{\epsilon n}}$ such that $v_i \in G'_{v_i}$ and there exists $x_i \in G'_{v_i}$ such that $\mathsf{dist}(v_i,x_i) \leq r$ and $(x_i,G'_{v_i})$ embeds $(o,\T_\ell)$.

For each $v_i$, let $(\zeta^{v_i}_t)_{t \geq 0}$ be the contact process on $G'_{v_i}$, started from only $v_i$ infected, and built using the same family of Poisson processes as the original contact process on $G$. If $i \neq j$, then $(\zeta^{v_i}_t)$ and $(\zeta^{v_j}_t)$ are independent (since the $G'_{v_i}$ are disjoint) and moreover, $\xi^W_t \supseteq \xi^{W'}_t \supseteq \cup_{v \in W'}\; \zeta^{v}_t$ for all $t$.

Define the events
$$E_i = \left\{|\zeta^{v_i}_{R\ell} | \geq \alpha^\ell\right\},\qquad 1 \leq i \leq \epsilon n.$$
We then have $P_{G_n}\Ll[E_i\Rr] \geq \sigma$. Thus, by standard large deviation estimates for binomial random variables (see Lemma \ref{l:largedev}), we have $$P_{G_n}\Ll[\sum_{i=1}^{\epsilon n} \1_{E_i} \geq \frac{\sigma}{2}\epsilon n\Rr] \geq 1 - e^{-c(\epsilon,\sigma)n}.$$

Finally, if the event in the above probability occurs, we have
$$|\xi^W_{R\ell}| \geq \alpha^\ell\cdot \sum_{i=1}^{\epsilon n} \1_{E_i}  \geq  \alpha^\ell \cdot \frac{\sigma}{2}\epsilon n> k\epsilon n.$$
This completes the proof.
\end{proof}

The proof of Theorem~\ref{t:regen} goes by showing first a seemingly weaker property, where the embedded trees are allowed to be somewhat damaged. We now introduce the relevant notions.

We call \emph{pruned} $\ell$-\emph{tree} any graph obtained by removing an edge from $\T_\ell$ and taking the connected component of the root. We call \emph{rooted pruned} $\ell$-\emph{tree} any rooted graph $(o,H)$, where $H$ is a pruned tree obtained from $\T_\ell$ and $o$ is the root of $\T_\ell$ (which belongs to $H$ by construction).

%\begin{defi}
We say that a rooted graph $(\rho,H)$ is $(\ell,r)$-\emph{favourable} if there exists $x \in H$ such that the following two conditions hold:
\begin{itemize}
\item the distance between $x$ and $\rho$ is $r$,
\item the rooted graph $(x,H)$ embeds a rooted pruned $\ell$-tree.
\end{itemize}
%\end{defi}
We say that a set of vertices $W \subset V_n$ is $(\ell,r)$-\emph{good} if there exists a family $(G'_v)_{v \in W}$ of subgraphs of $G_n$ whose vertex sets are pairwise disjoint such that for every $v \in W$, $G'_v$ contains $v$ and $(v,G'_v)$ is $(\ell,r)$-favourable.

\begin{prop}[Finding large good subsets]
\label{p:geom}
For $k$ and $r$ sufficiently large and for every $\ell$, there exists $\eps_0 > 0$ such that for every $\eps \le \eps_0$, the following holds with $\P_n$-probability tending to~$1$ as $n$ tends to infinity. From every $W \subset V_n$ of cardinality $k \eps n$, one can extract an $(\ell,r)$-good subset of cardinality $\eps n$.
\end{prop}

We say that a set of vertices $W \subset V$ is $r$-\emph{prepared} if the following two properties hold:
\begin{itemize}
\item for every $v \in W$, the $r$-neighbourhood of $v$ is loop-free (we call $r$-neighbour\-hood of a site $v \in V$ the set of sites $w$ whose distance to $v$ is at most~$r$),
\item for every two distinct $v,w \in W$, the $r$-neighbourhoods of $v$ and $w$ are disjoint.
\end{itemize}
The key step of the proof lies in the following proposition.
\begin{prop}[Key estimate]
\label{p:key}
For every $k > 1$ and $\eta < 2-4/k$, there exists $r_0 \in \N$ such that the following holds. For every $r \ge r_0$ and $\ell \in \N$, there exists $c > 0$ and $\eps_0 > 0$ such that for every $\eps \le \eps_0$ and every $n$ sufficiently large, the $\P_n$-probability that the set $\X = \{1,\ldots,k \eps n\} \subset V_n$ is $r$-prepared but does not contain any $(\ell,r)$-good subset of size $\eps n$ is smaller than
$$
(c \eps^\eta)^{|\X|}.
$$
\end{prop}
Before proving this key estimate, we introduce some terminology. A \emph{semi-graph} $g = (V,\mcl{E},\mcl{H})$ is a triple consisting of a set $V$ of vertices, a set $\mcl{E}$ of edges between points of $V$, and a set $\mcl{H}$ of \emph{half-edges}, each half-edge being attached to some vertex in $V$. Given two half-edges $h$ and $h'$, we write $h+h'$ to denote the edge obtained by ``gluing together'' the half-edges $h$ and $h'$ (that is to say, if $h$ is attached to a vertex $v$ and $h'$ to a vertex $w$, then $h+h'$ is an edge connecting $v$ and $w$). The distance in the semi-graph $g$ is simply the distance in the graph $(V,\mcl{E})$.

Consider the semi-graph $g = (V_n, \mcl{E}, \mcl{H})$ such that $\mcl{E} = \emptyset$ and each vertex in $V_n$ has exactly $d+1$ half-edges. We can construct a random regular graph with distribution $\P_n$ by the following recursive procedure. Take an arbitrary half-edge $h$ in $\mcl{H}$ (call it the \emph{elected} half-edge); take a half-edge $h'$ uniformly at random in $\mcl{H} \setminus \{h\}$; add the edge $h+h'$ to the set $\mcl{E}$, remove the half-edges $h$ and $h'$ from $\mcl{H}$; repeat until the set of half-edges is empty. (Recall that we assume $(d+1)n$ to be even, so that there is an even number of half-edges to begin with.) The resulting random graph has law $\P_n$. A feature of this procedure that will be crucial in our reasoning is that at each step, we have the freedom to choose the elected half-edge as we wish among the half-edges of $\mcl{H}$. 

It will be convenient to write operations on sets such as those done in the above construction in the more symbolic form
\begin{equation}
\label{update}
\mcl{E} \leftarrow \mcl{E} \cup \{h+h'\}, \qquad \mcl{H} \leftarrow \mcl{H} \setminus\{h,h'\}.
\end{equation}

\begin{proof}[Proof of Proposition~\ref{p:key}]
We take the semi-graph $g = (V_n, \mcl{E}, \mcl{H})$ such that $\mcl{E} = \emptyset$ and such that every site has exactly $(d+1)$ half-edges. As long as there is a half-edge $h$ attached to a site at distance strictly less than $r$ from $\X$, we choose this as the elected edge, pick another edge $h'$ uniformly at random from $\mcl{H} \setminus \{h\}$, and do the operations in \eqref{update}. When there is no longer any such half-edge, the graph is still incomplete, but we can already decide if the set $\X$ is $r$-prepared, since we have constructed the $r$-neighbourhood of $\X$ (and nothing more). If the set is not $r$-prepared, we can stop. On the event that the set is $r$-prepared, we continue our construction of the graph, with the aim to show that $\X$ will contain a good subset of size $\eps n$ with high probability. 

Let $\mcl{F}$ be the set of sites still having $(d+1)$ half-edges at this point. We call elements of $\mcl{F}$ \emph{fresh} vertices (they are still unseen by our construction of the graph). Let us call an element of $\X$ a \emph{seed}, and a site at distance $r$ from a seed a \emph{bud}. Recall that since we consider only the event that $\X$ is $r$-prepared, to every seed corresponds exactly $(d+1)d^{r-1}$ buds (so in total we have $|\X|(d+1)d^{r-1}$ buds). We say that a bud is \emph{active} if it has $d$ half-edges. We say that a seed is \emph{active} if it is associated with an active bud (i.e.\ if it is at distance $r$ from an active bud). A seed or a bud that is not active will be called \emph{quiet}.
As of now, every bud is active, and a fortiori every seed is active. We let $\rho$ be any active seed (as of now, this only means that we take $\rho \in \X$), and run the \textsf{Pass} as described below.

\medskip

\noindent\rule[0.5ex]{\linewidth}{1pt}

\noindent \textsf{The Pass} started from the active seed $\rho$.

\medskip

By definition of being an active seed, $\rho$ is associated with at least one active bud, say $x$, whose set of $d$ attached half-edges we write $\ov{\mcl{H}}$. We let $\ov{V} = \{x\}$, $\ov{\mcl{E}} = \emptyset$.

\medskip

\noindent \textsf{Step 1}

Let $v \in \ov{V}$ be at distance strictly less than $\ell$ from $x$ as measured in the graph $(\ov{V}, \ov{\mcl{E}})$, and such that there is a half-edge $h \in \ov{\mcl{H}}$ attached to it. If such a $v$ does not exist, then declare that the \textsf{Pass} is a \emph{success} and stop. Otherwise, pick $h'$ uniformly at random in $\mcl{H} \setminus \{h\}$, and let $v'$ be the vertex to which it is attached. 

If $v' \notin \mcl{F}$, then declare that a \emph{collision} is found. Say that it is a \emph{short collision} if moreover, $v' \in \ov{V}$; that it is a \emph{long collision} otherwise. If either a short collision is found, or it is the second time during the \textsf{Pass} that a long collision is found, then declare that the \textsf{Pass} is a \emph{failure}, do the updates displayed in \eqref{update} and stop. Else, go to \textsf{Step 2}.

If on the other hand $v' \in \mcl{F}$, then do the following updates
$$
\ov{V} \leftarrow \ov{V} \cup \{v'\}, \qquad \ov{\mcl{E}} \leftarrow \ov{\mcl{E}} \cup \{h + h'\},
$$
and go to \textsf{Step 2}.

\medskip

\noindent \textsf{Step 2}

Do the following updates
$$
\mcl{H} \leftarrow \mcl{H} \setminus \{h, h'\} , \qquad \ov{\mcl{H}} \leftarrow \ov{\mcl{H}} \setminus \{h\}, \qquad \mcl{F} \leftarrow \mcl{F} \setminus \{v'\},
$$
and go to \textsf{Step 1} (it never happens that $h' \in \ov{\mcl{H}}$ at this step, since we forbid short collisions, so there is no need to remove it from $\ov{\mcl{H}}$).

\noindent\rule[0.5ex]{\linewidth}{1pt}

\medskip

We then iterate this pass, always starting with a new active seed, until every seed is inactive or has already been used as the starting point of a pass.

%This algorithm is designed to discover an $(h,r)$-good subset of $\X$ ``as fast as possible''. 

We now list some important observations concerning the effect of passes.
\begin{obs}[A successful pass gives a favourable graph]
\label{pass_gives_favourable}
Write $B_r(\rho)$ for the $r$-neighbour\-hood of $\rho$ (in $g$). If the pass is a success, then the graph obtained by taking the union of $B_r(\rho)$ (seen as a sub-graph of $g$) and $(\ov{V}, \ov{\mcl{E}})$, rooted at $\rho$, is $(\ell,r)$-favourable. Indeed, this graph satisfies that
\begin{itemize}
\item $x$ is in the vertex set,
\item the distance from $x$ to $\rho$ is $r$,
\item the rooted graph $(x,(\ov{V}, \ov{\mcl{E}}))$ is isomorphic to a rooted pruned $\ell$-tree.
\end{itemize}
Moreover, such an $(\ell,r)$-favourable graph does not intersect any possible $(\ell,r)$-favourable graph obtained from previous passes.
\end{obs}

\begin{obs}[Many fresh sites]
\label{many_fresh}
During one pass, at most 
$$
c_\ell := d + d^2 + \cdots + d^\ell
$$
iterations of step 1 are performed, so in particular at most $c_\ell$ sites are removed from the set $\mcl{F}$ of fresh sites. Initially, there are 
$$
n-|\X|\Ll[ 1 + (d+1) + \cdots + (d+1) d^{r-1} \Rr]
$$
fresh vertices. Since we start with $|\X|$ active seed vertices, we can run the pass no more than $|\X|$ times. At any given time, there are thus always at least
$
n - c_{r,\ell} |\X|
$
fresh vertices, where
$$
c_{r,\ell} := 1 + (d+1) + \cdots + (d+1) d^{r-1} + c_\ell.
$$
\end{obs}
\begin{obs}[Many passes]
\label{many_passes}
Apart from the bud $x$ that is being explored, at most 2 active buds can be turned quiet during a pass. Indeed, a bud other than $x$ can be turned quiet during a pass only when a collision occurs, since none of the buds ever belongs to $\mcl{F}$. A collision can turn only one bud quiet at a time, and we allow for at most 2 collisions per pass. As a consequence, if we start (as we do) in a configuration where every bud is active and run the pass $t$ times, then at most 
$$
t + \frac{2t}{(d+1)d^{r-1}}
$$
seeds can be turned quiet, since initially every seed has $(d+1)d^{r-1}$ active buds. In other words, since we start with $|\X|$ seeds with all associated buds active, we can run the pass at least $\gamma_r |\X|$ times without running out of active seeds, where
\begin{equation}
\label{defgamma}
\gamma_r = \Ll(1 + \frac{2}{(d+1)d^{r-1}}\Rr)^{-1}.
\end{equation}
The crucial point for our reasoning is that we can take $\gamma_r$ arbitrarily close to $1$ by choosing $r$ sufficiently large.
\end{obs}

We now show that the probability to find many short collisions is very small. At any given time, there are no more than $d^\ell$ half-edges in $\ov{\mcl{H}}$. The probability to find a short collision while running step 1 of the pass is thus bounded by
$$
\frac{d^\ell}{(d+1) (n-c_{r,\ell} |\X|)} = \frac{d^\ell}{(d+1)(1-k c_{r,\ell} \eps )n},
$$
using Observation~\ref{many_fresh} on the number of fresh sites.

As said in Observation~\ref{many_fresh}, we run step 1 of the pass at most $c_\ell$ times, so the probability to find a short collision during one pass is bounded by
$$
\P\Ll[\mathsf{Bin}\Ll(c_\ell,\frac{d^\ell}{(d+1)(1-k c_{r,\ell} \eps )n}\Rr) \ge 1\Rr].
$$
As long as we choose $\eps$ sufficiently small that $k c_{r,\ell} \eps < 1$, we see that this probability is bounded by $\td{c}/n$ for some $\td{c}$ depending on $k$, $r$ and $\ell$. Since we run at most $|\X|$ passes, the probability to find at least $\eps n$ short collisions in total is bounded by
$$
\P\Ll[ \mathsf{Bin}\Ll(|\X|, \td{c}/n\Rr) \ge \eps n \Rr].
$$
We then use Lemma~\ref{l:largedev} with $m = |\X| = k \eps n$, $p = \td{c}/n$ and $p+\delta = 1/k$ to obtain that the probability above is bounded by
$$
\exp\Ll( -|\X| \Ll[\frac{1}{k} \log\Ll( \frac{n}{k\td{c}} \Rr) + \Ll(1-\frac1k\Rr) \log\Ll( \frac{1-1/k}{1-\td{c}/n} \Rr) \Rr]\Rr),
$$
which is much smaller than any given $(c \eps^\eta)^{|\X|}$ when $n$ is sufficiently large.

Let us say that a \emph{double collision} occurs when two long collisions are found during one pass. We now turn to the estimation of the probability that there are many double collisions. It follows from Observation~\ref{many_fresh} that the probability to find a long collision while running step 1 of the pass is bounded by 
$$
\frac{(d+1)c_{r,\ell} |\X|}{(d+1) (n-c_{r,\ell}|\X|)} = \frac{c_{r,\ell} |\X|}{(n-c_{r,\ell}|\X|)},
$$
Hence, the probability that a double collision is found during a pass is smaller than
$$
\P\Ll[\mathsf{Bin}\Ll(c_\ell,\frac{c_{r,\ell} |\X|}{(n-c_{r,\ell}|\X|)}\Rr) \ge 2 \Rr],
$$
where we used the fact that during one pass, we run step 1 at most $c_{\ell}$ times. Recall that $|\X| = k \eps n$, and note that
$$
\P\Ll[\mathsf{Bin}\Ll(c_\ell,\frac{k \eps c_{r,\ell}}{1-k \eps c_{r,\ell}}\Rr) \ge 2 \Rr] \sim \frac{c'}{2} \eps^2
$$
as $\eps$ tends to $0$, where
$$
c' = \Ll({k c_\ell c_{r,\ell}}\Rr)^2.
$$
Hence, we can choose $\eps$ sufficiently small (in terms of $k$, $r$ and $\ell$) so that 
\begin{equation}
\label{e:prob_fail}
\begin{array}{c}
\text{the probability for a pass to find a double collision is bounded by } c' \eps^2.
\end{array}
\end{equation}

The probability that we meet at least $\gamma_r |\X|-2\eps n$ double collisions in total is thus bounded by
$$
\P\Ll[ \mathsf{Bin}\Ll( |\X|, c' \eps^2 \Rr) \ge \gamma_r |\X| - 2 \eps n \Rr] = \P\Ll[ \mathsf{Bin}\Ll( k  \eps n, c' \eps^2 \Rr) \ge (k \gamma_r  - 2) \eps n \Rr],
$$
using \eqref{e:prob_fail} and the fact that we run at most $|\X|$ passes (and the fact that $|\X| = k \eps n$ for the equality). Applying Lemma~\ref{l:largedev} with $m = k  \eps n = |\X|$ , $p = c' \eps^2$ and 
$$
p+\delta = \gamma_r-\frac{2}{k},
$$
we see that this probability is bounded by
$$
\exp\Ll(- |\X| \Ll[ \Ll(\gamma_r-\frac2k\Rr) \log\Ll( \frac{\gamma_r-2/k}{c'\eps^2} \Rr)  + \Ll(1-\gamma_r+\frac2k\Rr) \log\Ll( \frac{1-\gamma_r+2/k}{1-c'\eps^2}\Rr)   \Rr] \Rr).
$$
By choosing $r$ sufficiently large, we can have $\gamma_r$ as close to $1$ as we wish. In particular, we can make sure that
$$
2 \Ll(\gamma_r-\frac2k \Rr) > \eta,
$$
so that the probability above is bounded by $(c\eps^{\eta})^{|\X|}$ for some constant $c$ independent of $\eps$.

To sum up, we have shown that in our procedure, we find
\begin{itemize}
\item at most $\eps n$ short collisions and
\item at most $\gamma_r |\X| - 2 \eps n$ double collisions
\end{itemize}
with probability larger than $1-2(c\eps^\eta)^{|\X|}$. On this event, since we run the pass at least $\gamma_r |\X|$ times by Observation~\ref{many_passes}, we see that at least $\eps n$ passes have been successful. This is exactly what we need to conclude, as was explained in Observation~\ref{pass_gives_favourable}.
\end{proof}
\begin{cor}
\label{c:key}
For every $k > 4$, every $r$ sufficiently large and every $\ell$, there exists $\eps_0 > 0$ such that for every $\eps \le \eps_0$, the following holds with $\P_n$-probability tending to $1$ as $n$ tends to infinity. From every $r$-prepared set $W \subset V_n$ of size $k \eps n$, one can extract an $(\ell,r)$-good subset of size $\eps n$.
\end{cor}
\begin{proof}
Proposition~\ref{p:key} ensures that there exists $\chi > 1$ such that for every $r$ sufficiently large and every $\ell$, there exists $c > 0$ such that for every $\eps > 0$ sufficiently small, the probability for a given $r$-prepared set $W \subset V_n$ of size $k \eps n$ not to contain any $(\ell,r)$-good subset of size $\eps n$ is smaller than $(c \eps^{\chi})^{k \eps n}$. We prove the corollary by a union bound, counting the total number of sets $W \subset V_n$ of size $k\eps n$. This number is
$$
{n \choose k\eps n} \le \frac{n^{k\eps n}}{(k \eps n)!}.
$$
Using the fact that $\log n! \ge n \log n - n$ (which can be proved by induction on $n$), we see that
$$
\log {n \choose k\eps n} \le k \eps n \log n - k \eps n \log(k \eps n) + k \eps n = k \eps n\Ll(1+\log\Ll(\frac{1}{k\eps}\Rr)\Rr).
$$
If we choose $\eps > 0$ sufficiently small, we can ensure that
$$
(c \eps^{\chi})^{k \eps n} \exp\Ll[k \eps n\Ll(1+\log\Ll(\frac{1}{k\eps}\Rr)\Rr)  \Rr]
$$
tends to $0$ as $n$ tends to infinity (since $\chi > 1$), and this proves the result.
\end{proof}
\begin{proof}[Proof of Proposition~\ref{p:geom}]
Note first that the number of loops of size bounded by $r$ remains tight as $n$ tends to infinity (in order to see this, it suffices to check that the expectation of the number of such loops remains bounded). In particular, with probability tending to $1$, there are less than $\sqrt{n}$ loops of size bounded by $r$. We assume from now on that the event that there are no more than $\sqrt{n}$ loops of size bounded by $r$ is realized.

For any site $v \in V$, the size of its $r$-neighbourhood is bounded by
$$
1 + (d+1) + (d+1) d + \cdots + (d+1) d^{r-1} =: \ov{c}_r.
$$
We choose $k > 8 \ov{c}_r$, and observe that from any set of $k \eps n$ vertices, we can extract a subset of $k \eps n/\ov{c}_r$ vertices whose $r$-neighbourhoods are pairwise disjoint. Letting $k' = k/{2\ov{c}_r} > 4$ and using the fact that there are no more than $\sqrt{n}$ loops of size bounded by $r$, we see that for $n$ large enough, we can extract from any set of $k \eps n$ vertices an $r$-prepared subset of cardinality $k' \eps n$. The conclusion of the theorem then follows by Corollary~\ref{c:key}.
\end{proof}
\begin{proof}[Proof of Theorem~\ref{t:regen}]
Let $(o,H)$ be a rooted pruned $\ell$-tree. Recall that $H$ was obtained from $\T_\ell$ by removing an edge and taking the connected component of the root. Out of the $d$ sub-trees of $o$ in $\T_\ell$, only one is affected by this construction. In particular, there exists $(v,H')$ a rooted sub-tree of $H$ such that $v$ is a neighbour of~$o$ and $(v,H')$ is isomorphic to $(o,\T_{\ell-1})$. 

As a consequence, any $(\ell,r)$-good set of vertices is $(\ell-1,r+1)$-regenerative. Theorem~\ref{t:regen} thus readily follows from Proposition~\ref{p:geom}.
\end{proof}

%
%
%
%
%
%%%%%%%%%%%%%%%%%%%%%%%%%%%%%%%%%%%%%%%%%%%%%%%%%%%%%%%%%%%%%%
%%%%%%%%%%%%%%%%%%%%%%%%%%%%%%%%%%%%%%%%%%%%%%%%%%%%%%%%%%%%%%
%
%
%
\section{Subcritical regime}
\label{s:subcrit}
%We now turn to the proof of part (a) of Theorem~\ref{t:mainsup}. 
We start from the following consequence of well-known estimates for the subcritical contact process on trees.

\begin{prop}For any $\lambda < \lambda_1(\T)$, there exists $C > 0$ such that
$$\lim_{n \to \infty} \sup_{A \subset \T: |A| = n} P_{\T,\lambda}\Ll[\uptau^A_{\T} > C \log n\Rr] = 0.$$
\end{prop}
\begin{proof}
The proof will depend on the fact that, for any $\lambda < \lambda_1(\T)$, there exist $c_0, C_0 > 0$ such that
\begin{equation} \label{eq:cllig} 
\mathbb{E}_{\T,\lambda}\Ll[|\xi^o_t|\Rr] \leq C_0 e^{-c_0 t},\quad t \geq 0.
\end{equation}
This claim follows from putting together several results in Section I.4 of \cite{lig99}; here we will simply outline them. We need the functions $\phi_\lambda(\rho)$ and $\beta(\lambda)$ (where $\rho > 0$ is an extra variable which is not involved in the construction of the contact process); these functions are  defined respecively in equation (4.23), page 87, and (4.48), page  96. In Proposition 4.27(a) and (b) we respectively have
\begin{equation} \label{eq:clligD}
\phi_\lambda(\rho) = \phi_\lambda\Ll(\frac{1}{\rho d}\Rr)
\end{equation}
and
\begin{equation} \label{eq:clligA}
\mathbb{E}_{\T,\lambda}\Ll[|\xi^o_t|\Rr] \leq C_\lambda \cdot \phi_\lambda(1)^t,\quad t \geq 0,
\end{equation}
for some constant $C_\lambda > 0$. In Proposition 4.44(a), we have
\begin{equation}
\label{eq:clligB}
\rho < \rho',\;\phi_\lambda(\rho') \geq 1 \Longrightarrow \phi_\lambda(\rho) < \phi_\lambda(\rho').
\end{equation}
Regarding $\beta$, in Corollary 4.78, Theorem 4.83 and Theorem 4.130, we respectively have that
\begin{eqnarray}
&&\label{eq:clligE} \beta(\lambda_1) = \frac{1}{d},\\
&&\label{eq:clligC}\phi_\lambda(\beta(\lambda)) = \phi_\lambda\Ll(\frac{1}{\beta(\lambda)\cdot d}\Rr) = 1,\\
&&\label{eq:clligF} \beta \text{ is strictly increasing in } [0,\lambda_1].
\end{eqnarray}
Now, for $\lambda < \lambda_1$, (\ref{eq:clligE}) and (\ref{eq:clligF}) give $\beta(\lambda) < \frac{1}{d}$. Together with (\ref{eq:clligC}), this shows that $\phi_\lambda(\rho) = 1$ for some $\rho > 1$. Then, using (\ref{eq:clligB}), we get $\phi_\lambda(1) < 1$, and then (\ref{eq:clligA}) gives the desired equation (\ref{eq:cllig}).

Let us now show how (\ref{eq:cllig}) completes the proof of our proposition. Noting that $|\xi^A_t| = \left|\cup_{x \in A}\; \xi^x_t\right| \leq \sum_{x \in A}\left|\xi^x_t\right|$, we have
$$P_{\T,\lambda}\Ll[|\xi^A_t| \neq \varnothing\Rr] \leq E_{\T,\lambda}\Ll[|\xi^A_t|\Rr] \leq |A|\cdot M_t \leq C_0|A|e^{-c_0t}.$$
The proof is completed by taking $C = 2/c_0$.
\end{proof}

%We will merely show how this can be obtained from results that are stated and proved in \cite{lig99}.

Part (a) of Theorem \ref{t:mainsup} is a consequence of the above proposition and the following result: 
\begin{lem}
\label{lem:stdom}
For any finite graph $G = (V,E)$ with degree bounded by $d+1$, $A \subset V$ and $t > 0$,
$$P_{G,\lambda}\Ll[\uptau^A_G > t\Rr] \leq \sup_{B \subset \T: |B| = |A|} P_{\T,\lambda}\Ll[\uptau^B_{\T} > t\Rr].$$
\end{lem}
\begin{proof}
Since the contact process is unaffected by the presence of loops (edges that start and end at the same vertex), we assume that $G$ has none.

We will now recall the concept of universal covering of the graph $G$; we will construct from $G$ a new graph $\mathcal{T} = (\mathcal{V},\mathcal{E})$ with certain desirable properties. We start fixing a reference vertex $x \in V$. We say that a sequence $\gamma = (\vec{e}_1,\ldots,\vec{e}_n)$ of oriented edges of $\vec{E}$ is a \textit{non-backtracking path from $x$} if $v_0(\vec{e}_1) = x$ and, for $1 \leq i < n $, $v_1(\vec{e}_i) = v_0(\vec{e}_{i+1})$ and $u(\vec{e}_i) \neq u(\vec{e}_{i+1})$ (recall that, for an oriented edge $\vec{e}$, $v_0(\vec{e}),\;v_1(\vec{e})$ and $u(\vec{e})$ respectively denote the starting vertex, ending vertex and undirected edge associated to $\vec{e}$). Let $\mathcal{V}$ be the set of all non-backtracking paths from $x$, including an empty path which we denote by $o$. For any $\gamma,\gamma'\in \mathcal{V}$ with $\gamma = (\vec{e}_1,\ldots, \vec{e}_n)$ and $\gamma' = (\vec{e}_1, \ldots, \vec{e}_n, \vec{e}_{n+1})$, we connect $\gamma$ and $\gamma'$ by an edge; this defines the edge set $\mathcal{E}$ of $\mathcal{T}$. Finally, put $\psi(o) = x$ and, for $\gamma = (\vec{e}_1,\ldots,\vec{e}_n)$, put $\psi(\gamma) = v_1(\vec{e}_n)$, the ending vertex of the path $\gamma$. It is now easy to check that $\mathcal{T}$ and $\psi$ satisfy the properties:
\begin{itemize}
\item[(a)] $\mathcal{T}$ is a tree with degree bounded by $d+ 1$;
\item[(b)] for every $\gamma \in \mathcal{V}$, $\psi$ maps the neighbourhood of $\gamma$ bijectively to the set
$$\{\vec{e}: \vec{e} \in \vec{E} \text{ has $\psi(\gamma)$ as its starting vertex}\}$$
\end{itemize}
(in case $\psi(\gamma)$ is connected to each of its neighbours by a single edge, property (b) just says that $\psi$ maps the neighbourhood of $\gamma$ bijectively to the neighbourhood of~$\psi(\gamma)$).

For $y \in V$, the set $\psi^{-1}(y)$ is called the \textit{fiber} of $y$. Define the set of configurations of $\{0,1\}^\mathcal{V}$ that have at most one particle per fiber,
$$\Omega_\mathcal{T} = \left\{\zeta \in \{0,1\}^\mathcal{V}: \sum_{\gamma \in \psi^{-1}(v)} \zeta(\gamma) \in \{0,1\}\right\}.$$
Define the projection $\pi: \Omega_\mathcal{T} \to \{0,1\}^V$ by $[\pi(\zeta)](v) = \sum_{\gamma \in \psi^{-1}(v)} \zeta(\gamma)$ for $v \in V$. We abuse notation and, for a set $B \subset \mathcal{V}$, we write $B \in \Omega_\mathcal{T}$ if $\1_B \in \Omega_\mathcal{T}$.

Given $B \in \Omega_\mathcal{T}$, we define the \textit{constrained contact process on $\mathcal{T}$}, $(\zeta^B_t)_{t\geq 0}$, as follows. We set $\zeta^B_0 = B$ and let $\zeta$ evolve as a contact process on $\mathcal{T}$ with the restriction that we suppress every transition which would result in a configuration not in $\Omega_\mathcal{T}$, that is, births on vertices belonging to fibers already containing infected vertices. Formally, $(\zeta_t)$ has generator
$$\mathcal{L}f(\zeta) = \sum_{\gamma \in \mathcal{V}} \Ll[f(\zeta^{0 \to \gamma}) - f(\zeta)\Rr] + \lambda \cdot \sum_{\substack{\{\gamma, \gamma'\} \in \mathcal{E}:\\\zeta(\gamma)=1}}\Ll[ \1_{\{\zeta^{1 \to \gamma'} \in \Omega_\mathcal{T}\}}\cdot \Ll(f(\zeta^{1 \to \gamma'}) - f(\zeta) \Rr)\Rr],$$
where $\zeta^{i \to \gamma'}$ is the configuration obtained by modifying $\zeta$ so that $\gamma'$ is set to state $i$. Noting that $G$ is finite, and hence $\zeta$ has at most $|V|$ infected vertices at any given time, it is not hard to see that the above generator indeed gives rise to a Feller process on $[0,\infty)^{\Omega_\mathcal{T}}$. In fact, $\zeta$ can be constructed from a Harris system on $\mathcal{T}$, with the extra care of ignoring any transmission mark which would cause a fiber to become doubly occupied. Also, using property (b) of $\mathcal{T}$ and $\psi$ stated above, it is easy to show that $(\pi(\zeta^B_t))_{t \geq 0}$ has the same distribution as the contact process on $G$ started from $\pi(B)$.

Now, if we start from $A \subset V$, we can choose an arbitrary $B \in \Omega_\mathcal{T}$ such that $\pi(B) = A$ and conclude that $\uptau^A_G$ has the same distribution as $\inf\{t:\zeta^B_t = \varnothing\}$. By seeing $\mathcal{T}$ (and hence $B$) as a subset of $\T$, we have that $\zeta^B$ is stochastically dominated by the contact process on $\T$ started from $B$ infected, and hence $\inf\{t: \zeta^B_t = \varnothing\}$ is stochastically dominated by $\uptau^B_\T$. This completes the proof.
\end{proof}

Let $G$ be a graph and $x$ a vertex of $G$. Given the contact process on $G$ started from a single infection at $x$, $(\xi^x_t)_{t\geq 0}$, define
$$\upkappa^x_G = \sup\{\mathsf{dist}(x,y):y \in \xi^x_t \text{ for some } t \geq 0\}.$$

\begin{lem}
For any finite graph $G$ with degree bounded by $d+1$, any vertex $x$ of $G$ and $k > 0$,
$$\mathbb{P}\Ll[\upkappa^x_G > k\Rr] \leq \mathbb{P}\Ll[\upkappa^o_\T > k\Rr].$$
\end{lem}
\begin{proof}
Repeating the construction in the previous lemma, we note that, if $\psi(\gamma) = x$ and $\psi(\gamma') = y$, then $\mathsf{dist}_\mathcal{T}(\gamma, \gamma') \geq \mathsf{dist}_G(x,y)$, and we then see that $\upkappa^x_G$ is stochastically dominated by $\sup\{\mathsf{dist}_\mathcal{T}(\gamma,\gamma'): \gamma' \in \zeta^{\{\gamma\}}_t \text{ for some } t \geq 0 \}$. The latter is in turn stochastically dominated by $\kappa^o_\T$, completing the proof.
\end{proof}

\begin{proof}[Proof of Theorem \ref{t:univsub}.]
In order to show that the contact process on $G$ dies out, it suffices to show, for any vertex $x$, that $\P\Ll[\uptau^{x}_G < \infty\Rr] = 1.$ Denote by $B_G(x,r)$ the subgraph of $G$ induced by the set of vertices at graph distance less than $r$ from $x$. Then,
$$\P\Ll[\uptau^x_G < \infty\Rr] = \lim_{r \to \infty} \P\Ll[\upkappa^x_G \leq r\Rr] = \lim_{r \to \infty} \P\Ll[\upkappa^x_{B_G(x,2r)} \leq r\Rr] \geq \lim_{r \to \infty} \P \Ll[ \upkappa^x_\T \leq r\Rr] = 1.$$

\end{proof}

\end{document}